\documentclass[11pt]{amsart}
\usepackage{amssymb}
\usepackage{amsfonts}
\usepackage{amsmath}
\usepackage{graphicx}
\usepackage{xcolor}
\usepackage{mathrsfs}
\usepackage{stmaryrd}
\usepackage{epsfig,color}
\usepackage{blindtext}
\usepackage{enumerate}
\usepackage{hyperref}
\usepackage{url}
\usepackage{bbm}
\usepackage{filecontents}
\usepackage{nicefrac,mathtools}
\usepackage{bm}  
\usepackage[nocompress]{cite}
\DeclareGraphicsExtensions{.pdf,.jpeg,.png}
\usepackage{epstopdf}
\usepackage{cancel} 
\usepackage[normalem]{ulem} 
\usepackage{verbatim} 
\usepackage{enumitem} 
\usepackage{tikz-cd}
\usetikzlibrary{cd}
\usepackage{color}
\usepackage[msc-links, lite]{amsrefs}
\usepackage{geometry}
\usepackage{tikz}
\usetikzlibrary{decorations.markings}
\usetikzlibrary{arrows.meta}

\usepackage{extarrows}
\newtheorem{theorem}{Theorem}[section]

\newtheorem{proposition}[theorem]{Proposition}
\newtheorem{lemma}[theorem]{Lemma}
\newtheorem{corollary}[theorem]{Corollary}

\newtheorem{fact}[theorem]{Fact}

\theoremstyle{definition}
\newtheorem{definition}[theorem]{Definition}

\newtheorem{remark}[theorem]{Remark}

\numberwithin{equation}{section}

\numberwithin{equation}{section}

\makeatletter
\let\save@mathaccent\mathaccent
\newcommand*\if@single[3]{%
\setbox0\hbox{${\mathaccent"0362{#1}}^H$}%
\setbox2\hbox{${\mathaccent"0362{\kern0pt#1}}^H$}%
\ifdim\ht0=\ht2 #3\else #2\fi
}

\makeatother

\makeatletter
\newcommand*{\transpose}{%
{\mathpalette\@transpose{}}%
}
\newcommand*{\@transpose}[2]{%
\raisebox{\depth}{$\m@th#1\intercal$}%
}
\makeatother
\usetikzlibrary{hobby}

 \usetikzlibrary{decorations}
\makeatletter
\def\pgfutil@Repeat#1#2{#2\ifnum#1>0
  \expandafter\pgfutil@firstofone\else\expandafter\pgfutil@gobble\fi
  {\expandafter\pgfutil@Repeat\expandafter{\the\numexpr#1-1\relax}{#2}}}
\tikzset{
  dash between/.code args={#1 and #2}{%
    \tikz@addoption{%
      \pgfgetpath\currentpath
      \pgfprocessround{\currentpath}{\currentpath}%
      \pgf@decorate@parsesoftpath{\currentpath}{\currentpath}%
      \pgfmathsetlengthmacro\firstpart{(#1)*\pgf@decorate@totalpathlength}%
      \pgfmathsetlengthmacro\secondpart{(#2-(#1))*\pgf@decorate@totalpathlength}%
      \pgfmathsetlengthmacro\thirdpart{(1-(#2))*\pgf@decorate@totalpathlength}%
      \edef\thirdpart{{\thirdpart}{0pt}}%
      \edef\firstpart{{\firstpart}{0pt}}%
      \pgfmathsetlengthmacro\secondpartlength{\pgfkeysvalueof{/tikz/dash between on}
                                            +(\pgfkeysvalueof{/tikz/dash between off})}%
      \pgfmathtruncatemacro\repetitions{\secondpart/\secondpartlength}%
      \pgfmathsetlengthmacro\secondexpand{\secondpart/\repetitions-\secondpartlength}%
      \edef\secondexpand{\the\dimexpr\pgfkeysvalueof{/tikz/dash between off}+\secondexpand\relax}%
      \edef\secondpart{%
        \pgfutil@Repeat{\the\numexpr\repetitions-1\relax}%
          {{\pgfkeysvalueof{/tikz/dash between on}}{\secondexpand}}%
      }%
      \edef\tikz@temp{\firstpart\secondpart\thirdpart}%
      \expandafter\pgfsetdash\expandafter{\tikz@temp}{+0pt}%
    }
  }
}
\makeatother
\tikzset{
  dash between style/.is choice,
  dash between style/dotted/.style        ={dash between on=\pgflinewidth,dash between off=2pt},
  dash between style/densely dotted/.style={dash between on=\pgflinewidth,dash between off=1pt},
  dash between style/loosely dotted/.style={dash between on=\pgflinewidth,dash between off=4pt},
  dash between style/dashed/.style        ={dash between on=3pt,dash between off=2pt},
  dash between style/loosely dashed/.style={dash between on=3pt,dash between off=6pt},
  dash between style/densely dashed/.style={dash between on=3pt,dash between off=2pt},
  dash between style/no/.style={dash between on=0pt, dash between off=1pt},
  dash between on/.initial=\pgflinewidth,
  dash between off/.initial=2pt,
  middle dotted line/.style={
    thick,
    dash between=.35 and .65}}

\newcommand\QQ{\mathbb{Q}}
\newcommand\CC{\mathbb{C}}

\newcommand\ZZ{\mathbb{Z}}

\newcommand\FF{\mathbb{F}}

\DeclareMathOperator{\homo}{Hom}

\DeclareMathOperator{\et}{\acute{e}t}

\def\holim{\qopname\relax m{holim}}

\usepackage[all]{xy}

\newcommand{\Addresses}{{
  \bigskip
  \footnotesize
  Ruida Di, \textsc{Morningside Center of Mathematics and Chinese Academy of Sciences, 55 Zhongguancun E Rd, Beijing 100190, China.}\par\nopagebreak
  \textit{E-mail address}, \texttt{drdmath@amss.ac.cn}

  Runjie Hu, \textsc{Department of Mathematics, 3368 Texas A\&M University, College Sta, TX 77843}\par\nopagebreak
  \textit{E-mail address}, \texttt{ runjie.hu@tamu.edu}

  Siqing Zhang, \textsc{Department of Mathematics, Yale University, 219 Prospect St, New Haven, CT 06511}\par\nopagebreak
  \textit{E-mail address}, \texttt{siqing.zhang@yale.edu}
}}

\author{Ruida Di, Runjie Hu, Siqing Zhang}

\title{A necessary condition for liftings of positive characteristic varieties with finite fundamental groups}

\begin{document}

\maketitle

\begin{abstract}
In this paper, we introduce a necessary condition for the existence of characteristic zero liftings of certain smooth, proper varieties in positive characteristic, using étale homotopy theory and Wall's finiteness obstruction. For a variety with finite étale fundamental group $\pi$, we define a notion of mod-$l$ finite dominatedness based on the $\FF_l$-chain complex of the universal cover of its $l$-profinite étale homotopy type. We prove that such a variety $X$ can be lifted to characteristic zero only if the above chain complex of $X$ is quasi-isomorphic to a bounded complex of finitely generated projective $\FF_l [\pi]$-modules. To prove this result, we extend Wall's discussions of finiteness obstructions to $l$-profinite complete spaces with finite fundamental group.
\end{abstract}

\section{Introduction}

Some characteristic $p>0$ varieties fail to lift to characteristic zero. Known obstructions include the failure of Kodaira vanishing and certain behavior of the fundamental group \cite{esnault2021obstruction, haboush1993varieties, totaro2019failure}. In this paper, we introduce a new necessary condition to such liftings, using étale homotopy theory and a variant of Wall’s finiteness obstruction. It is based on four observations:
\begin{enumerate}
    \item If a smooth proper variety $X$ over an algebraically closed field $k$ of characteristic $p>0$ admits a characteristic zero lifting $\widetilde{X}$ over a discrete valuation ring $R$, then the $l$-adic \'etale homotopy type $X^{\wedge}_{\et,l}$ is homotopy equivalent to $(\widetilde{X}_K)_{\et,l}$, where $K$ is the algebraic closure of the fraction field of $R$ for any $l\neq p$.
    \item For a complex variety $X$, its $l$-adic \'etale homotopy type $X^{\wedge}_{\et,l}$ is homotopy equivalent to its $l$-profinite completion $\widehat{X}_l$.
    \item Any proper complex variety is a finite CW complex.
    \item A CW complex with fundamental group $\pi$ is homotopy equivalent to a finite CW complex if and only if it is finitely dominated and the Wall finiteness obstruction in the reduced algebraic $K$-group $\widetilde{K}_0(\ZZ [\pi])$ vanishes.
\end{enumerate}

We introduce the terminologies and our results below.

\underline{\textit{$l$-profinite completion}.}

Let $l$ be a prime number. A $\pi$-($l$-)finite space is a space whose homotopy groups are finite ($l$-)groups.
Given a pro-space $X$, its $l$-profinite completion $X_l$ is a pro-$\pi$-$l$-finite-space which is initial among such pro-spaces that receives a map from $X$.
Recall that $\pi_1(X_l)$ is the $l$-profinite completion of $\pi_1(X)$.
If $X$ is a finite CW complex and simply-connected, then  $\pi_i(X_l)$ is the $l$-completion of $\pi_i(X)$. Two pro-spaces are said to be $l$-adic weak equivalent if the homotopy limits of their $l$-profinite completions are homotopy equivalent.

\underline{\textit{\'Etale homotopy type.}}

The étale homotopy type $X_{\et}$ of a scheme $X$ is a pro-$\pi$-finite space that generalizes  $\pi_1^{\et}(X)$ and $H^i_{\et}(X)$.
For complex varieties, the pro-space $X_{\et}$
  is weak equivalent to the profinite completion of the analytic space $X(\CC)^{an}$.
Coupled with the specialization isomorphism, it follows that if a smooth proper variety $X$ in characteristic $p$ lifts to characteristic zero, then its étale homotopy type $X_{\et}$ must be $l$-adic homotopy equivalent to a finite CW complex.

\underline{\textit{Wall finiteness obstruction.}}

Wall \cite{Wall-Finiteness-I,Wall-Finiteness-II} shows that a CW complex $Y$ is homotopy equivalent to a finite CW complex if and only if it is finitely dominated and the so-called Wall finiteness obstruction $\chi(Y)\in \widetilde{K}_0(\ZZ[\pi_1(Y)])$ vanishes. 

\underline{\textit{Finiteness obstruction for local spaces.}}

Finiteness obstruction for fiberwise $T$-local spaces with $T$ a set of primes was partially studied in \cite{local-finiteness}. We make a more geometric definition for $T$-local finite dominatedness (see Definition \ref{Def: local finite dominatedness}). That is, a fiberwise $T$-local CW complex $Y$ is $T$-local finitely dominated if $Y$ is the homotopy retract of the fiberwise $T$-localization of a finite CW complex. Then we prove that this geometric definition is equivalent to the notion of ``T-type FP'' in \cite{local-finiteness}*{Definition 3.2} (see Proposition \ref{Prop: Alternative Definition of Finite Dominatedness}),  which should be thought of as a chain complex definition of $T$-local finite dominatedness. 

Then one may define $T$-local Wall finiteness obstruction $\chi_T(Y)$ in the reduced algebraic $K$-group $\widetilde{K}_0(\ZZ_{(T)}[\pi_1])$ for any $T$-local finitely dominated space $Y$  (\cite{local-finiteness}*{Definition 3.3} or see Definition \ref{Def: local Wall obstruction}). We prove the following.

\begin{theorem}
\begin{enumerate}
    \item (Theorem \ref{Thm: Local Finiteness}) Let $Y$ be a connected, fiberwise $T$-local CW complex. Then $Y$ is fiberwise $T$-local homotopy equivalent to a finite CW complex if and only if $Y$ is $T$-local finitely dominated and the $T$-local Wall finiteness obstruction $\chi_{T}(Y)\in \widetilde{K}_0(\ZZ_{(T)}[\pi])$ vanishes.
    \item (Theorem \ref{Thm: Realization of local finiteness obstruction}) For any element $\sigma\in\widetilde{K}_0(\ZZ_{(T)}[\pi])$ with $\pi$ finitely presented, there exists a fiberwise $T$-local, $T$-local finitely dominated CW complex $Y$ with fundamental group $\pi$ such that $\chi_{T}(Y)=\sigma$.
\end{enumerate}
\end{theorem}

The item (1) is stated without proof in \cite{local-finiteness}*{Proposition 3.2}.
The item (2) is new.

\underline{\textit{Finiteness obstruction for complete spaces with finite fundamental group.}}

We make clear the notion of mod-$l$-finite dominatedness for any $l$-profinite complete space with finite fundamental group (see Definition \ref{Def: mod-l finitely dominated}). A connected CW complex $Y$ with $\pi=\pi_1(Y)^{\wedge}_l$ a finite $l$-group is mod-$l$ finitely dominated if the mod-$l$ chain complex of the universal cover of the $l$-profinite completion of $Y$ is quasi-isomorphic to a bounded chain of finitely generated projective modules over $\FF_l [\pi]$.

Note that in this setting, the Wall obstruction vanishes automatically, because a finite $l$-group $\pi$ has the trivial reduced $K$-group $\widetilde{K_0}(\FF_l[\pi])$.

Even for the simply-connected case, an $l$-profinite complete ``finite type'' space might not be $l$-adic weak equivalent to a finite CW complex (see \cite{Victor-Wilkerson-Counterexample-p-completion}*{Lemma 3.4 and the paragraph above
it} for an example). A requirement is that the space needs to be the $l$-profinite completion of some $l$-local space. This suggests the definition of $l$-local liftability for $l$-profinite complete spaces with finite fundamental groups (see Definition \ref{Def: l-local lifting}). Then the finiteness obstruction for the $l$-profinite complete case is the following:

\begin{theorem}[see Proposition \ref{Lemma: a necessary condition for finiteness obstruction}, Theorem \ref{Thm: Main Theorem 1}]\label{Thm: Main Intro}
Let $Y$ be a connected CW complex with $\pi=\pi_1(Y)^{\wedge}_l$ finite.
\begin{enumerate}
    \item If $Y$ is $l$-adic weak equivalent to a finite CW complex, then it is mod-$l$ finitely dominated.
    \item $Y$ is $l$-adic weak equivalent to a finite CW complex with the same fundamental group if and only if $Y$ is mod-$l$ finitely dominated and its $l$-profinite completion $Y_l$ is $l$-local liftable.
\end{enumerate}
\end{theorem}

We should emphasize that the mod-$l$ finite dominatedness is more subtle than only requiring the homology to be finite dimensional. One example is that a free resolution of $\FF_l$ as $\FF_l[\ZZ/l]$-modules is not perfect, in the sense that it is not chain homotopy equivalent to a bounded chain complex of finitely generated projective $\FF_l[\ZZ/l]$-modules, though its homology is $\FF_l$ at degree $0$.

\underline{\textit{Applications.}} 

To apply the previous topological discussion to algebraic geometry, we first prove the following of homotopical finiteness for schemes with finite \'etale cohomology and finite \'etale fundamental groups.

\begin{theorem}[see Corollary \ref{cor: finite gen piq}]\label{Thm: Main Intro II}
Let $X$ be a connected, pointed, locally Noetherian scheme. Assume that $\varprojlim (\pi^{\et}_1(X))^{\wedge}_l$ is a finite group, where $(\pi^{\et}_1(X))^{\wedge}_l$ is the pro-group of the $l$-profinite completion of the \'etale fundamental (pro-)group $\pi^{\et}_1(X)$. Assume that $H^{q}_{\et}(X;A)$ is finite for any $q$ and for any locally constant sheaf $A$ of abelian finite $l$-groups, then the homotopy group $\pi_q(X^{\wedge}_{\et,l})$ of the $l$-profinite completion $X^{\wedge}_{\et,l}$ of its \'etale homotopy type $X_{\et}$ is a finitely generated $\widehat{\ZZ}_l$-module for any $q\geq 2$.
\end{theorem}

Mod-$l$-finite dominatedness can be analogously defined for pro-spaces with finite fundamental group (see Definition \ref{defn: pro mod l finite dom}).
Then we define a scheme $X$ with  $(\pi_1^{\et}(X))^{\wedge}_l$ finite to be mod-$l$ finitely dominated if the $l$-profinite completion of its étale homotopy type is. The following is a direct corollary of Theorems \ref{Thm: Main Intro} and \ref{Thm: Main Intro II},

\begin{theorem}[see Theorem \ref{Cor: Cor 1}, Corollary \ref{cor: last cor}] 
Let $X$ be a connected, pointed, locally Noetherian scheme with $\pi=(\pi^{\et}_1(X))^{\wedge}_l$ finite.
\begin{enumerate}
    \item $X_{\et}$ is $l$-adic weak equivalent to a complex variety with fundamental group $\pi$ if and only if $X$ is mod-$l$ finitely dominated and $X^{\wedge}_{\et,l}$ is $l$-local liftable.
    \item If $X$ is smooth, proper over a separably closed field of characteristic $p>0$ with $l\neq p$, then a necessary condition for $X$ to lift to zero characteristic is that $X$ is mod-$l$ finitely dominated.
\end{enumerate}
\end{theorem}

Whether this obstruction of ``mod-$l$ finitely dominated'' is nontrivial in concrete cases remains an open question.

In the appendix, we provide an algebraic proof of a representation theory result (Corollary \ref{Cor:: freeness of modules}), which is alternatively indicated by topological discussions of $l$-complete Wall finiteness obstructions (see Remark \ref{Rem: vanishing of k-group}).  \\

\noindent\textbf{Acknowledgments.} 

We would like to thank Irina Bobkova, James F. Davis, Guozhen Wang, Xingting Wang, Sarah Witherspoon and Zhizhang Xie for useful conversations and comments. 
 R. H. is supported by
NSF Grant 2247322.
S. Z. is partially supported by the National Science Foundation under Grant No. DMS-1926686, and an AMS-Simons travel grant.
We are grateful to the referee who provides many useful comments in this manuscript.

\subsubsection*{Terminologies and Notations}

\begin{itemize}[leftmargin=0.25in]
    \item For a CW complex $Z$, the space $Z_l$ is Sullivan's $l$-profinite completion \cite{Sullivan-adams-conjecture}. 
    For a simply-connected and finite type $Z$, Bousfield's $l$-completion \cite{Bousfield-localization-of-spaces}, Bousfield-Kan's $\ZZ/l$-localization \cite{Bousfield-Kan} and Sullivan's $l$-profinite completion are equivalent, see \cite[\S2.3]{Barthel-Bousfield-Comparison-of-p-completion}.
    \item A pro-space is an object in the pro category of the homotopy category of pointed, connected CW complexes. A map of pro-spaces is a morphism in the pro category.
    \item For a pro-space $Y$, $Y^{\wedge}_l$ is the pro-space by Artin-Mazur's $l$-profinite completion \cite{artin-mazur-etale-homotopy}.
    \item For a (pro-)group $G$ (such as $\pi^{\et}_1(X)$), $G^{\wedge}_l$ is the \textit{pro-group} of the $l$-profinite completion of $G$. For an abelian group $A$, $A_l$ is the $l$-completion of $A$, which is the inverse limit of $A/l^nA$ for all $n$. If $A$ is a finitely generated abelian group, $A^{\wedge}_l$ is canonically isomorphic to $A_l$. 
    \item The pro-category of finite groups is naturally equivalent to the category of profinite groups, i.e. compact Hausdorff totally disconnected topological groups, see e.g. \cite[Prop. 3.2.12]{Lurie-rational-p-profinite}. Therefore, we will not differ pro systems of finite groups and profinite groups.
    \item Let $R$ be a unital ring (possibly noncommutative). A chain complex of left $R$-modules is perfect if it is chain homotopy equivalent to a bounded chain complex of finitely generated projective left $R$-modules.
    \item (\cite{artin-mazur-etale-homotopy}*{Theorem 4.3}, \cite{morel1993quelques}*{Theorem 2.4.1}) A pro-map $f:X\rightarrow Y$ of pro-spaces is an $l$-adic weak equivalence if one of the following equivalent conditions is satisfied:
\begin{enumerate}
    \item $\widehat{f}_l:\widehat{X}_l\rightarrow \widehat{Y}_l$ induces an isomorphism on pro homotopy groups;
    \item $\holim\widehat{f}_l:\holim\widehat{X}_l\rightarrow \holim\widehat{Y}_l$ is a homotopy equivalence of CW complexes;
    \item $f_*:\pi_1(X)^{\wedge}_l\rightarrow \pi_1(Y)^{\wedge}_l$ is an isomorphism and, for any local system of finite abelian $l$-group $M$ over $Y$, $f$ induces an isomorphism $H^*(Y;M)\rightarrow H^*(X;M)$.
    \item $f$ induces an isomorphism $H^*((Y)^{\wedge}_l;\ZZ/l)\rightarrow H^*((X)^{\wedge}_l;\ZZ/l)$.
\end{enumerate}
\end{itemize}

\section{Finiteness Obstruction for Local Spaces}\label{Section: Local Finiteness}

First recall Wall's result on when a connected CW complex $Y$ is homotopy equivalent to a finite CW complex (see, e.g., \cite{Wall-Finiteness-I}\cite{Wall-Finiteness-II}\cite{Davis-Surgery}\cite{Rosenberg-algebraic-K-theory}\cite{Steve-Ranicki-Survey-Finiteness}\cite{Weinberger-Statified-spaces}). 

Let $R$ be a ring. Let $C_*$ be a perfect chain complex of left $R$-modules, i.e., $C_*$ is chain homotopy equivalent to a bounded chain complex $P_*$ of finitely generated projective $R$-modules. The \textbf{Wall finiteness obstruction} of $C_*$ is the alternating sum $\chi(C_*)=\sum_{i=0}^{\infty}(-1)^i[P_i]$ in the reduced $K$-theory $\widetilde{K}_0(R[\pi])$. This invariant of $C_*$ is independent of the choice of $P_*$.

A CW complex $Y$ is \textbf{finitely dominated} if it is a homotopy retract of a finite CW complex, i.e., there exists a finite CW complex $Z$ together with maps $i:Y\rightarrow Z$ and $r:Z\rightarrow Y$ such that $r\circ i$ is homotopic to $Id_Y$. This is equivalent to saying that the cellular chain complex $C_*(\widetilde{Y};\ZZ)$ of the universal cover $\widetilde{Y}$ of $Y$ is perfect. Wall proves the following theorem.

\begin{fact}[\cite{Wall-Finiteness-I}\cite{Wall-Finiteness-II}]
A connected CW complex $Y$ is homotopy equivalent to a finite CW complex if and only if it is finitely dominated and the Wall finiteness obstruction $\chi(C_*(\widetilde{Y};\ZZ))\in \widetilde{K}_0(\ZZ[\pi_1(Y)])$ vanishes.
\end{fact}

For self-containedness and applications to Section \ref{Section: Complete Finiteness}, we discuss the finiteness obstruction for local spaces in this section. Theorem \ref{Thm: Local Finiteness} has beenobtained in \cite{local-finiteness} but Proposition/Theorem \ref{Prop: Alternative Definition of Finite Dominatedness}, \ref{Thm: Realization of local finiteness obstruction} are not included in this reference. The proofs are fiberwise local modifications for the integral case (e.g., see \cite{pedersen2017wallsfinitenessobstruction}).

One approach to localize non-simply-connected spaces is the fiberwise localization in \cite{Bousfield-Kan}*{Chapter I, Section 8} by Bousfield-Kan. Explicitly, for a CW complex $X$, one may consider the fibration $\widetilde{X}\rightarrow X\rightarrow K(\pi_1(X),1)$ and the fiberwise $l$-localization produces a fibration $\widetilde{X'}\rightarrow X'\rightarrow K(\pi_1(X),1)$ such that $\widetilde{X}'$ is the $l$-localization of the $\widetilde{X}$. Notice that fiberwise localization is functorial (\cite{Bousfield-Kan}*{p.~41, (ii)}).

Let $T$ be a set of primes and let $\ZZ_{(T)}\subset \QQ$ be the localization of $\ZZ$ at $T$. 

\begin{definition}
Two CW complexes are \textbf{fiberwise $T$-local homotopy equivalen}t if their fiberwise $T$-localizations are homotopy equivalent.
\end{definition}

\begin{definition}
A connected CW complex is \textbf{fiberwise $T$-local} if its universal cover is $T$-local.
\end{definition}

The discussion of local finiteness obstructions relies heavily on the following construction and Lemma \ref{Lemma: Localization on homotopy groups}, \ref{Lemma: Main Tech}:

Let $Y,Z$ be CW complexes with a map $f:Z\rightarrow Y$ which induces an isomorphism $\pi_1$. Let $\pi$ be the fundamental group. Assume that $Y$ is fiberwise $T$-local. Let $Z'$ be the fiberwise $T$-localization of $Z$. Then $f$ induces a map $f':Z'\rightarrow Y$ and $\widetilde{Z'}$ is homotopy equivalent to $\widetilde{Z}_{(T)}$. Consider the following diagram.
\[
\begin{tikzcd}
  ... \arrow[r] & \pi_{i}(\widetilde{Z}) \arrow[r,"(\widetilde{f})_*"] \arrow[d] & \pi_{i}(\widetilde{Y}) \arrow[r] \arrow[d,] & \pi_{i}(\widetilde{Y},\widetilde{Z}) \arrow[r] \arrow[d] & ... \\
  ... \arrow[r] & \pi_{i}(\widetilde{Z})\otimes \ZZ_{(T)} \arrow[r,"(\widetilde{f})_{*}\otimes \ZZ_{(T)}"] \arrow[d,"\cong"] & \pi_{i}(\widetilde{Y})\otimes \ZZ_{(T)} \arrow[r] \arrow[d,"\cong"] & \pi_{i}(\widetilde{Y},\widetilde{Z})\otimes \ZZ_{(T)} \arrow[r] \arrow[d,dotted,"\cong"] & ... \\  
  ... \arrow[r] & \pi_{i}(\widetilde{Z'})=\pi_{i}((\widetilde{Z})_{(T)}) \arrow[r,"(\widetilde{f'})_{*}"] & \pi_{i}(\widetilde{Y}) \arrow[r] & \pi_{i}(\widetilde{Y},\widetilde{Z'}) \arrow[r] &  ...
\end{tikzcd}
\]
Since localization is an exact functor, the second row is exact. The universal property of localization induces the dashed arrow. The five-lemma implies that 

\begin{lemma}\label{Lemma: Localization on homotopy groups}
The dashed arrow in the diagram is an isomorphism.    
\end{lemma}

Let $\alpha_j(j=1,...,k)$ be some elements in $\pi_i(Y,Z)\cong \pi_i(\widetilde{Y},\widetilde{Z})$ with $i\geq 2$. Then one may construct the union $W$ of $Z$ with $k$ copies of $i$-cells and a map $g:W\rightarrow Y$ such that the restriction of $g$ on $Z$ is $f$ and that on the $i$-cells correspond to $\alpha_j$'s each. Let $W'$ be the fiberwise $T$-localization of $W$. Let $\beta_j$'s be the images of $\alpha_j$ in $\pi_{i}(\widetilde{Y},\widetilde{Z'})$ as shown in the previous diagram. The following can be deduced easily from the properties of localization. 

\begin{lemma}\label{Lemma: Main Tech}
\begin{enumerate}
    \item  $H_i(\widetilde{W'},\widetilde{Z'};\ZZ)\cong \pi_i(\widetilde{W'},\widetilde{Z'})=\oplus_j \ZZ_{(T)}[\pi]$ and $H_q(\widetilde{W'},\widetilde{Z'};\ZZ)=0$ otherwise.
    \item $H_q(\widetilde{Y},\widetilde{W'};\ZZ)\cong H_{q}(\widetilde{Y},\widetilde{Z'};\ZZ)$ for $q\neq i,i+1$.
    \item There is an exact sequence
    \[
    0\rightarrow H_{i+1}(\widetilde{Y},\widetilde{Z'};\ZZ)\rightarrow H_{i+1}(\widetilde{Y},\widetilde{W'};\ZZ) \rightarrow H_{i}(\widetilde{W'},\widetilde{Z'};\ZZ)=\oplus_j \ZZ_{(T)}[\pi] \xrightarrow{\phi} H_{i}(\widetilde{Y},\widetilde{Z'};\ZZ)\rightarrow H_{i}(\widetilde{Y},\widetilde{W'};\ZZ) \rightarrow 0
    \]
    such that $\phi$ maps $1\in\ZZ_{(T)}[\pi]$ to the Hurewicz image of $\beta_j$. 
\end{enumerate}   
\end{lemma}

\begin{definition}\label{Def: local finite dominatedness}
A fiberwise $T$-local CW complex $Y$ is \textbf{$T$-local finitely dominated} if there exists a finite CW complex $K$ together with maps $Y\xrightarrow{i} K'\xrightarrow{r} Y$ such that $r\circ i$ is homotopy equivalent to $Id_Y$, where $K'$ is the fiberwise $T$-localization of $K$.   
\end{definition}

\begin{proposition}\label{Prop: Alternative Definition of Finite Dominatedness}
A fiberwise $T$-local CW complex $Y$ is $T$-local finitely dominated if and only if the followings hold:
\begin{enumerate}
    \item $\pi=\pi_1(Y)$ is finitely presented;
    \item the cellular chain complex $C_*(\widetilde{Y};\ZZ)$ is perfect over $\ZZ_{(T)}[\pi]$. 
\end{enumerate}
\end{proposition}

\begin{remark}
The chain complex condition is used in the definition of ``T-type FP'' in \cite{local-finiteness}*{Definition 3.2}. We prove here that this is equivalent to a geometric definition, which is analogous to the integral case.
\end{remark}

\begin{proof}
The ``only if'' part: 

By the proof of \cite{Wall-Finiteness-I}*{Lemma 1.3}, not only $\pi$ is dominated, but also the kernel of $r_*:\pi_1(K)\rightarrow \pi$ is normally finitely generated in $\pi_1(K)$. So we may add finitely many $2$-cells to $K$ and get maps $Y\xrightarrow{p} L\xrightarrow{q} Y$ such that $q\circ p$ is homotopy equivalent to $Id_Y$ and $p,q$ both induce isomorphisms on $\pi_1$, where $L$ is the union of $K$ and $2$-cells. Take the fiberwise $T$-localization $L'$ of $L$ and there are new maps $Y\xrightarrow{p'} L'\xrightarrow{q'} Y$. Then the item (2) directly follows from \cite{pedersen2017wallsfinitenessobstruction}*{Lemma 1.9}.

The ``if'' part:
We inductively construct $K$ and the map $r$. Assume that $C_*(\widetilde{Y};\ZZ_{(T)})$ is chain homotopy equivalent to a chain complex $P_*:0\rightarrow P_m\rightarrow ...\rightarrow P_0\rightarrow 0$ of finitely generated projective $\ZZ_{(T)}[\pi]$-modules.

First construct a $2$-dimensional finite CW complex $Z_1$ together with a map $Z_1\rightarrow Y$ which induces an isomorphism on $\pi_1$.

Inductively assume that we have constructed a finite CW complex $Z_n$ of dimension $n$ (except when $n=1$, the dimension is $2$) and a map $Z_n \rightarrow Y$ with $1\leq n\leq m-1$ such that $\pi_1(Z_n)\rightarrow \pi_1(Y)$ is an isomorphism and $\pi_q(Y,Z'_n)=0$ for any $q\leq n$, where $Z'_n$ is the fiberwise $T$-localization of $Z_n$. The map $Z_n\rightarrow Y$ factors through $Z'_n$. 

\hypertarget{Technical-Point}{(*)}: Then there is induced a chain map $C_*(\widetilde{Z'_n};\ZZ)\rightarrow P_*$. Let $D_*$ be its mapping cone, whose homology is isomorphic to $H_*(\widetilde{Y},\widetilde{Z'_n};\ZZ)$. By \cite{pedersen2017wallsfinitenessobstruction}*{Lemma 1.12}, $H_{n+1}(\widetilde{Y},\widetilde{Z'_n};\ZZ)\cong \pi_{n+1}(\widetilde{Y},\widetilde{Z'_n})$ is a finitely generated $\ZZ_{(T)}[\pi]$-module. 

By Lemma \ref{Lemma: Localization on homotopy groups} and the diagram above it, there exist finitely many elements $x_1,...,x_s\in\pi_{n+1}(Y,Z_n)\cong \pi_{n+1}(\widetilde{Y},\widetilde{Z_n})$ whose images generate the $\ZZ_{(T)}[\pi]$-module $\pi_{n+1}(\widetilde{Y},\widetilde{Z'_n})$. As in the construction above Lemma \ref{Lemma: Main Tech}, we can form $Z_{n+1}$ by the union of $Z_n$ and $s$-copies of $(n+1)$-cells together with a map $Z_{n+1}\rightarrow Y$ induced by $x_1,...,x_s$. Lemma \ref{Lemma: Main Tech} deduces that $\pi_q(\widetilde{Y},\widetilde{Z'_{n+1}})=H_q(\widetilde{Y},\widetilde{Z'_{n+1}};\ZZ)=0$ for $q\leq n+1$.

Continue this construction till $Z_{m-1}$. By \cite{pedersen2017wallsfinitenessobstruction}*{Lemma 1.12} and the mapping construction above, $H_q(\widetilde{Y},\widetilde{Z'_{m-1}};\ZZ)=0$ when $q\neq m$ and $P=H_m(\widetilde{Y},\widetilde{Z'_{m-1}};\ZZ)$ is a finitely generated projective $\ZZ_{(T)}[\pi]$-module. Let $Q$ be a finitely generated projective $\ZZ_{(T)}[\pi]$-module such that $P\oplus Q$ is free. Consider the following exact sequence:
\[
0\rightarrow P\otimes \ZZ_{(T)}[t,t^{-1}]\oplus Q\otimes \ZZ_{(T)}[t,t^{-1}] \xrightarrow{\psi_1=(1-t)\oplus 1} P\otimes \ZZ_{(T)}[t,t^{-1}]\oplus Q\otimes \ZZ_{(T)}[t,t^{-1}] \xrightarrow{\psi_2} P\rightarrow 0
\]
Notice that $P\otimes \ZZ_{(T)}[t,t^{-1}]\oplus Q\otimes \ZZ_{(T)}[t,t^{-1}]$ is a free $\ZZ_{(T)}[\pi\times \ZZ]$-module. Consider the map $Z_{m-1}\times S^1\rightarrow Y\times S^1$. Since the universal cover of $S^1$ is contractible, we conclude that $Y\times S^1$ is fiberwise $T$-local, that $(Z_{m-1}\times S^1)'=Z'_{m-1}\times S^1$ and that $H_*(\widetilde{Y\times S^1},\widetilde{Z'_{m-1}\times S^1};\ZZ)\cong H_*(\widetilde{Y},\widetilde{Z'_{m-1}};\ZZ)$. Let $y_1,...,y_s$ be a basis of $P\otimes \ZZ_{(T)}[t,t^{-1}]\oplus Q$. One may scale $y_1,...,y_s$ by some products of the primes not in $T$ so that $\psi_2(y_1),...,\psi_2(y_s)$ are in the image of $\pi_m(Y,Z_{m-1})=\pi_m(\widetilde{Y},\widetilde{Z_{m-1}})\rightarrow \pi_m(\widetilde{Y},\widetilde{Z'_{m-1}})=P$. Attach s copies of $m$-cells to $Z_{m-1}\times S^1$ corresponding to $\psi_2(y_1),...,\psi_2(y_s)$ and we obtain a map $W\rightarrow Y\times S^1$. By Lemma \ref{Lemma: Main Tech}, $H_q(\widetilde{Y\times S^1},\widetilde{W'};\ZZ)=0$ for $q\neq m$ and $H_q(\widetilde{Y\times S^1},\widetilde{W'};\ZZ)=P\otimes \ZZ_{(T)}[t,t^{-1}]\oplus Q\otimes \ZZ_{(T)}[t,t^{-1}]$.

Do the same operation for the map $\psi_1$ and we can get a finite CW complex $K$ and a map $K\rightarrow Y\times S^1$ such that the induced map $K'\rightarrow Y\times S^1$ is a homotopy equivalence. Then the natural inclusion and projection $Y\rightarrow Y\times S^1\simeq K\rightarrow Y$ proves the finite dominatedness of $Y$.
\end{proof}

\begin{definition}\label{Def: local Wall obstruction}
Let $Y$ be a fiberwise $T$-local, $T$-local finitely dominated CW complex with fundamental group $\pi$. Then the chain complex $C_*(\widetilde{Y};\ZZ_{(T)})$ is chain homotopy equivalent to a bounded chain complex $P_*$ of finitely generated $Z_{(T)}[\pi]$-modules. The \textbf{$T$-local Wall finiteness obstruction} is $\chi_T(Y)=\sum_{i=0}^{\infty}(-1)^i[P_i]\in \widetilde{K}_0(\ZZ_{(T)}[\pi])$.
\end{definition}

This definition is the same as \cite{local-finiteness}*{Definition 3.3}.

\begin{theorem}\label{Thm: Local Finiteness}
Let $Y$ be a connected, fiberwise $T$-local CW complex. Then $Y$ is fiberwise $T$-local homotopy equivalent to a finite CW complex if and only if $Y$ is $T$-local finitely dominated and the $T$-local Wall finiteness obstruction $\chi_{T}(Y)\in \widetilde{K}_0(\ZZ_{(T)}[\pi])$ vanishes.
\end{theorem}

\begin{remark}
This theorem is the same as \cite{local-finiteness}*{Proposition 3.2} but its proof is omitted in the reference.
\end{remark}

\begin{proof}
The ``only if'' part is obvious. For the ``if'' part: 

Like the previous proof, we may assume that $C_*(\widetilde{Y};\ZZ_{(T)})$ is chain homotopy equivalent to a chain complex $P_*:0\rightarrow P_m\rightarrow ...\rightarrow P_0\rightarrow 0$ of finitely generated projective $\ZZ_{(T)}[\pi]$-modules. With the same inductive construction, we have $Z_{m-1}\rightarrow Y$. Then $H_q(\widetilde{Y},\widetilde{Z'_{m-1}};\ZZ)=0$ when $q\neq m$ and $\pi_m(\widetilde{Y},\widetilde{Z'_{m-1}})=H_m(\widetilde{Y},\widetilde{Z'_{m-1}};\ZZ)$ is a finitely generated stably free $\ZZ_{(T)}[\pi]$-module due to the vanishing of $\chi_T(Y)$. Then we may take the union of $Z_{m-1}$ with a bouquet of some copies of $S^{m-1}$ to make $H_m(\widetilde{Y},\widetilde{Z'_{m-1}};\ZZ)$ free. Then by Lemma \ref{Lemma: Localization on homotopy groups}, we may find elements in $\pi_m(Y,Z_{m-1})$ such that their images in $\pi_m(\widetilde{Y},\widetilde{Z'_{m-1}})$ form a basis. Attach cells to $Z_{m-1}$ for these elements and obtain a map $Z_m\rightarrow Y$. Lemma \ref{Lemma: Main Tech} shows that this map is a fiberwise $T$-local homotopy equivalence.
\end{proof}

\begin{theorem}\label{Thm: Realization of local finiteness obstruction}
For any element $\sigma\in\widetilde{K}_0(\ZZ_{(T)}[\pi])$ with $\pi$ finitely presented, there exists a fiberwise $T$-local, $T$-local finitely dominated CW complex $Y$ with fundamental group $\pi$ such that $\chi_{T}(Y)=\sigma$.
\end{theorem}

\begin{proof}
$\sigma$ can be represented by a finitely generated projective $\ZZ_{(T)}[\pi]$-module $P$. Let $Q$ be a finitely generated projective $\ZZ_{(T)}[\pi]$-module such that $P\oplus Q$ is free. Consider the following exact sequence:
\[
0\rightarrow Q\oplus P\oplus Q\oplus ...\xrightarrow{\text{shifting}} P\oplus Q\oplus P\oplus Q\oplus ...\rightarrow P\rightarrow 0 
\]

Construct a finite CW complex $K$ with fundamental group $\pi$. Let $L$ be the union with countably many $m$-cells and $K$ so that $H_m(\widetilde{L'},\widetilde{K'};\ZZ)=P\oplus Q\oplus P\oplus Q\oplus ...$. We may attach $(m+1)$-cells to $L$ and form a CW complex $M$ so that $H_{m+1}(\widetilde{M'},\widetilde{L'};\ZZ)\rightarrow H_m(\widetilde{L'},\widetilde{K'};\ZZ)$ corresponds to the shifting map. Then $H_i(\widetilde{M'},\widetilde{K'};\ZZ)=0$ for $i\neq m$ and $H_i(\widetilde{M'},\widetilde{K'};\ZZ)=P$. Then consider the composition of maps $K\times S^1\rightarrow M\times S^1\rightarrow M'\times S^1$. Using the same technique in the last part of the proof of Proposition \ref{Prop: Alternative Definition of Finite Dominatedness}, one can prove that $M'\times S^1$ is fiberwise $T$-local homotopy equivalent to the union of $K\times S^1$ with finitely many cells. That is, $M'\times S^1$ is fiberwise $T$-local homotopy equivalent to a finite CW complex. Then $Y=M'$ is what we want.
\end{proof}

\section{Finiteness Obstruction for l-complete Space with Finite Fundamental Group}\label{Section: Complete Finiteness}

In this section, we discuss the finiteness obstructions for $l$-adic complete homotopy types with finite fundamental groups.

\begin{definition}
Let $Y$ be a connected CW complex with $\pi_1(Y)$ finite. $Y$ is \textbf{finite type} (\textbf{$l$-local finite type}, \textbf{$l$-complete finite type} resp.) if $\pi_q(Y)$ is a finitely generated module over $\ZZ$ ($\ZZ_{(l)}$, $\widehat{\ZZ}_l$ resp.) for any $q\geq 2$.
\end{definition}

\begin{definition}\label{Def: mod-l finitely dominated}
Let $Y$ be a connected CW complex with $\pi=\pi_1(Y)^{\wedge}_l$ a finite $l$-group. $Y$ is \textbf{mod-$l$ finitely dominated} if the chain complex $C_*((Y_l)^{\sim};\FF_l)$ of $\FF_l[\pi]$-modules is perfect, where $(Y_l)^{\sim}$ is the universal cover of the $l$-profinite completion $Y_l$.
\end{definition}

The following is a corollary of Fact \ref{Fact: Projective Modules are Free}.

\begin{lemma}
$C_*((Y_l)^{\sim};\FF_l)$ is perfect if and only if it is chain homotopy equivalent to a bounded chain complex of finitely generated free $\FF_l[\pi]$-modules.
\end{lemma}

The following is a special case of \cite{artin-mazur-etale-homotopy}*{p.~52, Theorem 4.11}. 

\begin{lemma}\label{Lem: Commutativity of universal cover and l-completion}
Let $Y$ be a connected CW complex with $\pi_1(Y)$ finite. Then the canonical map $(\widetilde{Y})_l\rightarrow (Y_l)^{\sim}$ is a homotopy equivalence.
\end{lemma}

\begin{lemma}\label{Lem: Finite dominated implies finite type}
If a connected CW complex $Y$ with $\pi_1(Y)$ finite is mod-$l$-finitely dominated, then $Y_l$ is $l$-complete finite type.
\end{lemma}

\begin{proof}
By definition, $H_q(\widetilde{Y};\ZZ/l)$ is finite dimensional for any $q$. Assume $Y^{\wedge}_l=\{Z_i\}_{i\in I}$. By Lemma \ref{Lem: Commutativity of universal cover and l-completion}, $\{\widetilde{Z_i}\}_{i\in I}$ is the $l$-adic completion of $\widetilde{Y}$, where each $\widetilde{Z_i}$ is the universal cover of $Z_i$. By \cite{artin-mazur-etale-homotopy}*{Theorem 4.3}, $H^q(\widetilde{Y};\ZZ/l)\cong \varinjlim_i H^q(\widetilde{Z_i};\ZZ/l)$ for any $q$. 

By the universal coefficient theorem, $H_q(\widetilde{Y};\ZZ/l)\cong \homo(H^q(\widetilde{Y};\ZZ/l),\ZZ/l)$ and $H_q(\widetilde{Z_i};\ZZ/l)\cong \homo(H^q(\widetilde{Z_i};\ZZ/l),\ZZ/l)$ since $H_q(\widetilde{Y};\ZZ/l)$ and $H_q(\widetilde{Z_i};\ZZ/l)$ are both finite dimensional. By Pontryagin duality, 
\begin{multline*}
\varprojlim_i \homo(H^q(\widetilde{Z_i};\ZZ/l),\ZZ/l)=\varprojlim_i \homo(H^q(\widetilde{Z_i};\ZZ/l),\QQ/\ZZ)\cong \homo(\varinjlim_i H^q(\widetilde{Z_i};\ZZ/l),\QQ/\ZZ) \\
= \homo(\varinjlim_i H^q(\widetilde{Z_i};\ZZ/l),\ZZ/l)
\end{multline*}

Hence, $H_q(\widetilde{Y};\ZZ/l)$$\cong \varprojlim_i H_q(\widetilde{Z_i};\ZZ/l)$ for any $q$. By \cite{Hu-Zhang-formal-manifold}*{Lemma 1.19}, $H_q((\widetilde{Y})_l;\ZZ/l)$ is finite for any $q$. By \cite{Lurie-rational-p-profinite}*{Theorem 3.4.12} or \cite{Hu-Zhang-formal-manifold}*{Theorem 1.5}, $(Y_l)^{\sim}\simeq( \widetilde{Y})_l$ is $l$-complete finite type. Then so is $Y_l$. 
\end{proof}

The possible Wall finiteness obstruction for $l$-profinite complete homotopy types in the reduced $K$-theory $\widetilde{K}_0(\FF_l[\pi])$ is vacuous for $\pi$ a finite $l$-group due to the following fact.

\begin{fact}\label{Fact: Projective Modules are Free}\label{cor: k0 is z}
Any finitely generated projective $\FF_{l^s}[\pi]$-module is free if $\pi$ is a finite $l$-group, where $\FF_{l^s}$ is the finite field of cardinality $l^s$.
In particular, the group $K_0(\FF_{l^s}[\pi])\cong \ZZ$ is generated by the free $\FF_{l^s}[\pi]$-module of rank $1$.
\end{fact}

\begin{proof}
It is known that the only simple left/right $\FF_{l^s}[\pi]$-module is the trivial module $\FF_{l^s}$ (see \cite{Serre-local-fields}*{Chapter IX, Theorem 2}). Then the fact follows from \cite{Serre-representation-finite-groups}*{Chapter 14, Proposition 4.1 and Corollary 1}.
\end{proof}

At first, it may seem that we only need to take care of an appropriate definition for the finite dominatedness in the $l$-adic setting.
However, there appears some other technical difficulty: even in the simply-connected case, not every $l$-profinite complete ``finite'' space can be realized by a finite CW complex (see \cite{Victor-Wilkerson-Counterexample-p-completion}*{p.~573 the last paragraph} for a counterexample). At least, the $l$-profinite complete spaces need to be the $l$-profinite completion of some $l$-local space. 

\begin{definition}\label{Def: l-local lifting}
An $l$-complete finite type CW complex $Y_l$ with finite fundamental group admits an \textbf{$l$-local lifting} if there exists a fiberwise $l$-local finite type CW complex $Z$ with the same fundamental group such that the fiberwise $l$-completion of $Z$ is homotopy equivalent to $Y_l$.
\end{definition}

\begin{lemma}
Let $Z$ be a connected CW complex with $\pi=\pi_1(Z)$ a finite $l$-group. Assume that either $Z$ is finite type or $l$-local finite type. Then there is a natural homotopy equivalence $Z_l\rightarrow Z'$, where $Z'$ is the fiberwise $l$-completion of $Z$.
\end{lemma}

\begin{proof}
Consider the following diagram the fiberwise $l$-completion of $Z$.
\[
\begin{tikzcd}
   \widetilde{Z} \arrow[r] \arrow[d] &  Z \arrow[r] \arrow[d] & K(\pi,1) \arrow[d,phantom, sloped, "="] \\
   \widetilde{Z'}\simeq(\widetilde{Z})_l \arrow[r] &  Z' \arrow[r] & K(\pi,1)
\end{tikzcd}
\]
By the finite-type assumption, $Z'$ is $l$-complete finite type and $\pi_1(Z')$ is a finite $l$-group. Therefore, $\widetilde{Z'}$ is $l$-profinite complete. So $\pi_q(Z')\cong \pi_q(\widetilde{Z'})\cong \pi_q(((Z')_l)^{\sim})\cong \pi_q((Z')_l)$ for any $q\geq 2$ due to Lemma \ref{Lem: Commutativity of universal cover and l-completion}.
Hence $Z'$ is also $l$-profinite complete. Then the map $Z\rightarrow Z'$ factors through $Z_l$. This induces a map $(Z_l)^{\sim}\rightarrow \widetilde{Z'}$ and this map is a homotopy equivalence due to Lemma \ref{Lem: Commutativity of universal cover and l-completion}. 
\end{proof}

A direct corollary is the following.

\begin{corollary}\label{Cor: completion = fiberwise completion + localization}
Let $Z$ be a finite type CW complex with $\pi_1(Z)$ a finite $l$-group. Then the $l$-profinite completion $Z_l$ of $Z$ is the fiberwise $l$-completion of the fiberwise $l$-localization of $Z$. 
\end{corollary}

\begin{proposition}\label{Lemma: a necessary condition for finiteness obstruction}
Let $Y$ be a connected CW complex with $\pi=\pi_1(Y)^{\wedge}_l$ finite. If $Y$ is $l$-adic weak equivalent to a finite CW complex, then $Y$ is mod-$l$ finitely dominated.  
\end{proposition}

\begin{proof}
Assume that $Y_l$ is the $l$-profinite completion of a finite CW complex $X$. Since $\pi_1(X)^{\wedge}_l\cong \pi_1(Y_l)$ is finite, there exists a finite covering space $X'$ of $X$ corresponding to the quotient $\pi_1(X)\rightarrow \pi_1(X)^{\wedge}_l$, together with a map $X'\rightarrow \widetilde{Y_l}$. That is, the pro-space of the $l$-profinite completion $(X')^{\wedge}_l$ is simply-connected. Then $H^*(X';\FF_l)\cong H^*((X')^{\wedge}_l;\FF_l)\cong H^*((\widetilde{Y})^{\wedge}_l;\FF_l)$ by the definition of $l$-adic weak equivalence. 

Assume that $(\widetilde{Y})^{\wedge}_l$ is the pro space $\{Z_i\}$. By definition, $H^q((\widetilde{Y})^{\wedge}_l;\FF_l)=\varinjlim_i H^q(Z_i;\FF_l)$. By the same argument as the second paragraph of the proof of Lemma \ref{Lem: Finite dominated implies finite type}, $\varprojlim_i H_q(Z_i;\FF_l)\cong \homo(\varinjlim_i H^q(Z_i;\FF_l),\FF_l)$. Then $\varprojlim_i H_q(Z_i;\FF_l)$ is finite dimensional. By \cite{Hu-Zhang-formal-manifold}*{Lemma 1.19}, $H^q(\widetilde{Y}_l;\FF_l)\cong \varinjlim_i H^q(Z_i;\FF_l)=H^q((\widetilde{Y})^{\wedge}_l;\FF_l)$.
 So the natural quasi-isomorphism $C_*(X';\FF_l)\rightarrow C_*(\widetilde{Y_l};\FF_l)$ proves that $Y$ is mod-$l$ finitely dominated.    
\end{proof}

\begin{theorem}\label{Thm: Main Theorem 1}
Let $Y$ be a connected CW complex with $\pi=\pi_1(Y)^{\wedge}_l$ finite such that its $l$-profinite completion $Y_l$ is $l$-complete finite type. Then $Y$ is $l$-adic weak equivalent to a finite CW complex with fundamental group $\pi$ if and only if $Y$ is mod-$l$ finitely dominated and $Y_l$ admits an $l$-local lifting.
\end{theorem}

\begin{proof}
For the ``only if'' part, Proposition \ref{Lemma: a necessary condition for finiteness obstruction} proves the first half. Now assume $\pi_1(X)=\pi$ and then the $l$-local liftability of $Y_l$ follows from Corollary \ref{Cor: completion = fiberwise completion + localization}. 

The ``if'' part can be directly deduced from the following proposition.  
\end{proof}

\begin{proposition}\label{prop: finiteness obstruction for the local case}
Let $Y$ be a connected, pointed, $l$-local finite type CW complex with $\pi=\pi_1(Y)$ a finite $l$-group. Then $Y$ is the fiberwise $l$-localization of a finite CW complex with fundamental group $\pi$ if and only if $Y$ is mod-$l$ finitely dominated.
\end{proposition}

\begin{proof}
The ``only if'' part is obvious. We are left with the ``if'' part. The proof is almost the same as Theorem \ref{Thm: Local Finiteness}. A difference in the inductive step is the paragraph \hyperlink{Technical-Point}{(*)}. 

Assume that $C_*(\widetilde{Y};\FF_l)$ is chain homotopy equivalent to a chain complex $F_*:0\rightarrow F_m\rightarrow ...\rightarrow F_0\rightarrow 0$ of finitely generated free $\FF_l[\pi]$-modules. Inductively assume that we have constructed a finite CW complex $Z_n$ of dimension $n$ (except when $n=1$) and a map $Z_n \rightarrow Y$ with $1\leq n\leq m-1$ such that $\pi_1(Z_n)\rightarrow \pi_1(Y)$ is an isomorphism and $\pi_q(Y,Z'_n)=0$ for any $q\leq n$, where $Z'_n$ is the fiberwise $l$-localization of $Z_n$. As $Y$ is $l$-local finite type, $H_{n+1}(\widetilde{Y},\widetilde{Z'_n};\ZZ)\cong \pi_{n+1}(\widetilde{Y},\widetilde{Z'_n})$ is a finitely generated $\ZZ_{(l)}$-module, in particular, finitely generated over $\ZZ_{(l)}[\pi]$. The rest inductive construction is the same as the proof of Theorem \ref{Thm: Local Finiteness}.

For the last step, we have constructed a finite CW complex $Z_{m-1}$ and a map $Z_{m-1}\rightarrow Y$ such that $\pi_1(Z_{m-1})\rightarrow \pi_1(Y)$ is an isomorphism and $\pi_q(Y,Z'_{m-1})=0$ for any $q\leq m-1$. Let $Z'_{m-1}$ be the fiberwise $l$-localization of $Z_{m-1}$.

The natural chain map $C_*(\widetilde{Z_{m-1}};\FF_l)\rightarrow C_*(\widetilde{Y};\FF_l)$ induces a chain map $C_*(\widetilde{Z_{m-1}};\FF_l)\rightarrow F_*$ since $C_*(\widetilde{Z_{m-1}};\FF_l)$ is free. Then its mapping cone $D_*$ has the form $0\rightarrow D_m\rightarrow ...\rightarrow D_0$, where each $D_i$ is a finitely generated free $\FF_l[\pi]$-module.

Since $H_*(\widetilde{Z_{m-1}};\FF_l)\cong H_*(\widetilde{Z'_{m-1}};\FF_l)$, $D_*$ is chain homotopy equivalent to $C_*(\widetilde{Y},\widetilde{Z'_{m-1}};\FF_l)$. So $H_q(D_*)=0$ for any $q\neq m$ and by \cite{pedersen2017wallsfinitenessobstruction}*{Lemma 1.12}, $H_m(\widetilde{Y},\widetilde{Z'_{m-1}};\FF_l)\cong H_m(D_*)$ is a finitely generated free $\FF_l[\pi]$-module. 

Since $\pi_{m-1}(\widetilde{Y},\widetilde{Z'_{m-1}})=H_{m-1}(\widetilde{Y},\widetilde{Z'_{m-1}};\ZZ)=0$, $\pi_m(\widetilde{Y},\widetilde{Z_{m-1}})\otimes \FF_l\cong \pi_m(\widetilde{Y},\widetilde{Z'_{m-1}})\otimes \FF_l\cong H_m(\widetilde{Y},\widetilde{Z'_{m-1}};\ZZ)\otimes \FF_l\cong H_m(\widetilde{Y},\widetilde{Z'_{m-1}};\FF_l)$ by the universal coefficient theorem.

Then choose elements $y_1,...,y_t$ in $\pi_m(Y,Z_{m-1})\cong \pi_m(\widetilde{Y},\widetilde{Z_{m-1}})$ such that their image in $H_m(\widetilde{Y},\widetilde{Z'_{m-1}};\FF_l)$ form a basis as a free $\FF_l[\pi]$-module. Then we can obtain $Z_m$ and a map $Z_m\rightarrow Y$ induced by $y_1,...,y_t$. Analogous to Lemma \ref{Lemma: Main Tech}, the long exact sequence of $\FF_l$-coefficient homologies for the triple $(\widetilde{Z'_{m-1}},\widetilde{Z'_m},\widetilde{Y})$ implies that $H_q(\widetilde{Y},(\widetilde{Z'_m});\FF_l)=0$ for any $q$. So $Z_m$ is a finite CW complex fiberwise $l$-local homotopy equivalent to $Y$.
\end{proof}

\begin{remark}\label{Rem: vanishing of k-group}
As a corollary of Proposition \ref{prop: finiteness obstruction for the local case} and Theorem \ref{Thm: Local Finiteness}, such a $Y$ is mod-$l$ finitely dominated if and only if $Y$ is $l$-local finitely dominated and the $l$-local Wall finite obstruction vanishes. This implies that $\widetilde{K}_0(\ZZ_{(l)}[\pi])=0$ for any finite $l$-group $\pi$. We will give a more elementary algebraic proof of this fact in Appendix.
\end{remark}

\section{Topological Finiteness of Varieties with Finite Fundamental Group}

We will apply the the previous results of finiteness obstruction to algebraic geometry. Before this, we establish some finiteness result for varieties in this section. Let $l$ be a prime number. 

\begin{definition}
A pointed, connected CW complex $Y$ is \textbf{$\pi$-$l$-finite} if each homotopy group $\pi_q(Y)$ is a finite $l$-group.    
\end{definition}

Let $\{Y_i\}_{i\in I}$ be a pro-space with each $Y_i$ $\pi$-$l$-finite. Assume that $\pi=\varprojlim_i \pi_1(Y_i)$ is finite. Let $\widetilde{Y}_i$ be the universal cover of $Y_i$.

\begin{lemma}\label{Lem: Univerisal Cover Cohomology}
If $\varinjlim_i H^q(Y_i;\widetilde{A})$ is finite for any $q$ and for any local coefficient $\widetilde{A}$ of finite $l$-groups for $\{Y_i\}$, then $\varinjlim_i H^q(\widetilde{Y}_i;\FF_l)$ is finite for any $q$.
\end{lemma}

\begin{proof}
Consider the pro system of finite $l$-groups $\{\pi_1(Y_i)\}_i$. Let $H_i$ be the image of $\pi\rightarrow \pi_1(Y_i)$ and let $Y_{H_i}$ be the corresponding covering space of $Y_i$. Then the covering map of pro-spaces $\{Y_{H_i}\}\rightarrow \{Y_i\}$ is a weak equivalence since this map induces an isomorphism on the pro homotopy groups. By \cite{artin-mazur-etale-homotopy}*{Theorem 4.3}, for any local coefficient $\widetilde{A}$ of finite $l$-groups for $\{Y_i\}$, $\varinjlim_i H^q(Y_{H_i};\widetilde{A})\cong \varinjlim_i H^q(Y_i;\widetilde{A})$ is also finite.

Hence, we may assume that each homomorphism in the pro system $\{\pi_1(Y_i)\}_i$ is a surjection. Then the canonical map $\pi\rightarrow \pi_1(Y_i)$ is also surjective for any $i$. Since the inverse limit $\pi$ is finite, there exists a cofinal subsystem $J\rightarrow I$ such that $\pi_1(Y_j)$ is isomorphic to $\pi$ for any $j\in J$ and each homomorphism in the pro system $\{\pi_1(Y_j)\}_{j\in J}$ is an isomorphism.

Now take $\widetilde{A}=\FF_l[\pi]$. Then $H^q(\widetilde{Y}_j;\FF_l)\cong H^q(Y_j;\widetilde{A})$ for any $q,j$ (see \cite{May-local-coefficient}*{Example 4}). So the lemma holds.
\end{proof}

\begin{corollary}\label{cor: pi q fin gen}
With the same assumption as Lemma \ref{Lem: Univerisal Cover Cohomology}, $\varprojlim_i \pi_q(Y_i)\cong \varprojlim_i \pi_q(\widetilde{Y}_i)$ is a finitely generated $\widehat{\ZZ}_l$-module for any $q\geq 2$.
\end{corollary}

\begin{proof}
This is reduced to the simply-connected case by passing to the universal covers of $Y_i$'s.
By \cite{Hu-Zhang-formal-manifold}*{Theorem 1.5}, the simply-connected case follows from the conclusion of Lemma \ref{Lem: Univerisal Cover Cohomology}.
\end{proof}

\begin{corollary}\label{cor: finite gen piq}
Let $X$ be a connected, pointed, locally Noetherian scheme. Let $(\pi^{\et}_1(X))^{\wedge}_l$ be the pro-group defined by the $l$-profinite completion.  Assume that $\varprojlim (\pi^{\et}_1(X))^{\wedge}_l$ is a finite group and that for any locally constant sheaf $A$ of abelian finite $l$-groups $H^{q}_{\et}(X;A)$ is finite for any $q$, then $\varprojlim \pi_q(X^{\wedge}_{\et,l})$ is a finitely generated $\widehat{\ZZ}_l$-module for any $q\geq 2$.
\end{corollary}
\begin{proof}
    \cite{artin-mazur-etale-homotopy}*{Corollary 10.7, 10.8} identifies the \'etale fundamental group (\'etale cohomology of locally constant sheaves resp.) with the fundamental group (twisted coefficient cohomology resp.) of the \'etale homotopy type. We then conclude this corollary using Corollary \ref{cor: pi q fin gen}.
\end{proof}

\begin{corollary}
Let $X$ be a variety over a separably closed field $k$ of characteristic $p\geq 0$. Let $l$ be a prime number. Assume that $(\pi_1^{\et}X)^{\wedge}_l$ is a finite group.
Assume either of the following is true:
\begin{enumerate}
    \item[(a)] $p\neq l$;
    \item[(b)] $X$ is proper over $k$.
\end{enumerate}
Then $\varprojlim\pi_q(X^{\wedge}_{\et,l})$ is a finitely generated $\widehat{\ZZ}_l$-module for any $q\geq 2$.
\end{corollary}
\begin{proof}
To apply Corollary \ref{cor: finite gen piq}, we need to show that $H^q_{\et}(X;A)$ is finite. For (a) this follows from \cite[Th. finitude, Cor. 1.10]{deligne1977sga4.5}.
For (b), this follows from
\cite[VI, Cor. 2.8]{Milne-Etale-cohomology}.
\end{proof}

\section{Application to Algebraic Varieties}

In this section, we apply previous results to algebraic varieties. We first need to generalize the notion of mod-$l$ finite dominatedness to pro-spaces.

\begin{definition}\label{defn: pro mod l finite dom}
Let $\{Y_i\}$ be a pro space with each $Y_i$ $\pi$-$l$-finite. Assume that $\pi=\varprojlim \pi_1(Y_i)$ is a finite $l$-group. Let $Y$ be the homotopy inverse limit of $\{Y_i\}$. $\{Y_i\}$ is \textbf{pro mod-$l$ finitely dominated} if the chain complex $C_*(\widetilde{Y};\FF_l)$ of $\FF_l[\pi]$-modules is perfect, where $\widetilde{Y}$ is the universal cover of $Y$.
\end{definition}

Then the following is a direct corollary of Theorem \ref{Thm: Main Theorem 1}.

\begin{theorem}\label{Thm: Main Theorem 2}
Let $\{Y_i\}$ be a pro system of $\pi$-$l$-finite CW complexes with $\pi=\varprojlim \pi_1(Y_i)$ a finite $l$-group. Assume that $\holim_{i} Y_i$ is $l$-complete finite type. $\{Y_i\}$ is $l$-adic weak equivalent to a finite CW complex with fundamental group $\pi$ if and only if $\{Y_i\}$ is pro mod-$l$ finitely dominated and $\holim_i Y_i$ is $l$-local liftable.    
\end{theorem}

The followings are the applications to algebraic geometry.

\begin{definition}\label{Def: Finite Dominatedness for Schemes}
Let $X$ be a connected, pointed, locally Noetherian scheme with $\pi=(\pi^{\et}_1(X))^{\wedge}_l$ a finite $l$-group. $X$ is \textbf{mod-$l$ finitely dominated} if the $l$-profinite completion $(X_{\et})^{\wedge}_l$ of the \'etale homotopy type $X_{\et}$ is.
\end{definition}

\begin{theorem}\label{Cor: Cor 1}
With the same assumption as Definition \ref{Def: Finite Dominatedness for Schemes}, $X_{\et}$ is $l$-adic weak equivalent to a complex variety with fundamental group $\pi$ if and only if $X$ is mod-$l$ finitely dominated and $X^{\wedge}_{\et,l}$ is $l$-local liftable.
\end{theorem}

\begin{proof}
By Theorem \ref{Thm: Main Theorem 2}, $X_{\et}$ is $l$-adic weak equivalent to a finite CW complex with the same fundamental group if and only if $X$ is mod-$l$ finitely dominated and $l$-local liftable. Following the construction in \cite{Deligne-Sullivan}*{paragraph below Lemma on p.~1082}, any finite CW complex is homotopy equivalent to a complex variety.
\end{proof}
Let $R$  be a discrete valuation ring with the residue field $k$ separably closed of characteristic $p>0$. Let $R\rightarrow \CC$ be a ring embedding. Let $Y$ be a smooth, proper variety over $R$.
Recall from \cite{artin-mazur-etale-homotopy}*{Corollary 12.12, Corollary 12.13} that $(Y_k)^{\wedge}_{\et,l}$ is canonically weak equivalent to $(Y_
{\CC})^{\wedge}_{\et,l}$ for any prime $l\neq p$. Proposition \ref{Lemma: a necessary condition for finiteness obstruction} and Theorem \ref{Cor: Cor 1} deduces the following, which is our main motivation for this paper.

\begin{corollary}\label{cor: last cor}
Let $X$ be a connected, proper, smooth variety over $k$. Let $l\neq p$ be a prime. Assume that $\pi=(\pi^{\et}_1(X))^{\wedge}_l$ is a finite $l$-group. 
\begin{enumerate}
    \item If $X$ has a lifting to a smooth, proper scheme $Y$ over $R$, then $X$ is mod-$l$ finitely dominated.
    \item If the $Y$ above satisfies that $\pi_1(Y^{an}_\CC)=\pi$, then  $X$ is mod-$l$ finitely dominated and $X^{\wedge}_{\et,l}$ is $l$-local liftable.
\end{enumerate}
\end{corollary}

\begin{remark}
The condition of item (2) actually deduces that $(\pi^{\et}_1(X))^{\wedge}_{l'}$ is trivial for any prime $l'\neq l,p$. 
\end{remark}

\appendix

\section{Some Algebraic Facts}

This appendix gives a more elementary proof of Remark \ref{Rem: vanishing of k-group}.
Lemma \ref{Lemma: appendix 2} in the case where $R$ is a discrete valuation ring can be directly deduced from \cite[Theorem 77.2]{curtis1966representation}.

By Nakayama's lemma, the following is easy.

\begin{lemma}\label{Lemma: Nakayama}
Let $R$ be a possibly non-commutative, local ring with the maximal $2$-sided ideal $m$. Let $D=R/m$ be the quotient division algebra. Let $M$ be a finitely generated left $R$-module. Then $\{x_i\}$ form a minimal generating set of $M$ over $R$ if and only if their representatives $\overline{x_i}$ in $M/mM$ form a basis over $D$.
\end{lemma}

\begin{lemma}\label{Lemma: appendix 2}
Let $R$ be a commutative local ring with the maximal ideal $m$. Let $\pi$ be a finite group. Let $M$ be a finitely generated projective $R[\pi]$-module. If  $\overline{M}=M/mM$ is a free $(R/m)[\pi]$-module, then $M$ is a free $R[\pi]$-module.
\end{lemma}

\begin{proof}
Since $R$ is local, $M$ is a finitely generated free $R$-module. Let $\{\overline{e_i}\}$ be a basis for $\overline{M}$ as an $R/m[\pi]$-module. Let $\pi=\{g_1,...,g_r\}$. Then $\{g_j\overline{e_i}\}$ is a basis for $\overline{M}$ as an $R/m$-module. 

Let $e_i\in M$ be a lift of $\overline{e_i}$. By Lemma \ref{Lemma: Nakayama}, $\{g_je_i\}$ is a minimal generating set of $M$ over $R$ and hence they form a basis of $M$ as an $R$-module. This shows that $M$ is a free $R[\pi]$-module.
\end{proof}

The following is a direct corollary of Lemma \ref{Lemma: appendix 2} and Fact \ref{Fact: Projective Modules are Free}.

\begin{corollary}\label{Cor:: freeness of modules}
Let $R$ be a commutative local ring. Let $\pi$ be a finite $l$-group. Assume that there exists a quotient homomorphism $R\rightarrow \FF_{l^s}$, where $\FF_{l^s}$ is the finite field of cardinality $l^s$. Then any finitely generated projective $R[\pi]$-module is free. In particular, $K_0(R[\pi])\cong \ZZ$ which is generated by the free $R[\pi]$-module of rank $1$.
\end{corollary}

\bibliographystyle{amsalpha}
\bibliography{ref}

\Addresses

\end{document}